\documentclass[11pt]{article}

\usepackage[utf8]{inputenc}
\usepackage[british]{babel}
\usepackage[a4paper, margin=2.5cm]{geometry}
\usepackage{amsmath, amsthm, amssymb, amsfonts}
\usepackage{mathtools}
\usepackage{thmtools}
\usepackage[shortlabels]{enumitem}
\usepackage[font={small, it}]{caption}
\usepackage{microtype}
\usepackage{xcolor}
\usepackage{mathabx}
\usepackage{hyperref}
\usepackage{bm}
\usepackage[normalem]{ulem}

\newcommand{\cA}{\ensuremath{\mathcal A}}
\newcommand{\cB}{\ensuremath{\mathcal B}}

\newcommand{\cS}{\ensuremath{\mathcal S}}
\newcommand{\cT}{\ensuremath{\mathcal T}}

\newcommand{\AP}{\ensuremath{\text{AP}}}

\renewcommand{\phi}{\varphi}
\renewcommand{\rho}{\varrho}

\let\setminus=\smallsetminus
\let\emptyset=\varnothing

\newcommand\redsout{\bgroup\markoverwith{\textcolor{red}{\rule[0.5ex]{2pt}{0.5pt}}}\ULon}

\declaretheorem[parent=section]{theorem}
\declaretheorem[sibling=theorem]{lemma}

\declaretheorem[sibling=theorem]{claim}

\setlist{itemsep=0.1em, topsep=0.1em, parsep=0.1em, partopsep=0.1em}

\hypersetup{
  colorlinks,
  linkcolor={red!60!black},
  citecolor={green!50!black},
  urlcolor={blue!80!black}
}

\colorlet{RoyalRed}{red!70!black}
\definecolor{RoyalBlue}{rgb}{0.25, 0.41, 0.88}
\definecolor{RoyalAzure}{rgb}{0.0, 0.22, 0.66}

\newlength{\bibitemsep}
\newlength{\bibparskip}
\let\oldthebibliography\thebibliography
\renewcommand\thebibliography[1]{%
  \oldthebibliography{#1}%
  \setlength{\parskip}{\bibitemsep}%
  \setlength{\itemsep}{\bibparskip}%
}

\title{Small subsets without $k$-term arithmetic progressions}
\author{
  Rajko Nenadov\thanks{Google Z\"urich. Email: \texttt{rajkon@gmail.com}.}
}
\date{}

\begin{document}
\maketitle

\begin{abstract}
Szemer\'edi's theorem implies that there are $2^{o(n)}$ subsets of $[n]$ which do not contain a $k$-term arithmetic progression. A sparse analogue of this statement was obtained by Balogh, Morris, and Samotij, using the hypergraph container method: For any $\beta > 0$ there exists $C > 0$, such that if $m \ge Cn^{1 - 1/(k-1)}$ then there are at most $\beta^m \binom{n}{m}$ $m$-element subsets of $\{1, \ldots, n\}$ without a $k$-term arithmetic progression. We give a short, inductive proof of this result. Consequently, this provides a short proof of the Szemer\'edi's theorem in random subsets of integers.
\end{abstract}

\section{Introduction}

The celebrated Szemer\'edi's theorem \cite{szemeredi75} states that for every $k \in \mathbb{N}$ and $\alpha > 0$, if $n$ is sufficiently large then any subset $S \subseteq [n]$ of size $|S| \ge \alpha n$ contains a $k$-term arithmetic progression (AP for short). This implies that there are $2^{o(n)}$ subsets of $[n]$ with no $k$-term AP. A better understanding of the $o(n)$ term was obtained only recently by Balogh, Liu, and Sharifzadeh \cite{balogh17}.

The following extension to small subsets was obtained by Balogh, Morris, and Samotij \cite{balogh2015independent} using the newly developed hypergraph containers method.

\begin{theorem} \label{thm:main}
    For every $\beta > 0$ there exist $\gamma, C > 0$, such that if $m \ge Cn^{1 - 1/(k-1)}$ then all but at most
    $$
        \beta^m \binom{n}{m}
    $$
    $m$-element subsets of $[n]$ contain at least $\gamma n^2 (m/n)^k$ $k$-term arithmetic progressions.
\end{theorem}

It should be noted that the result from \cite{balogh2015independent} provides a bound on the number of subsets without a $k$-term AP, but the same idea gives the counting version if applied with the hypergraph container theorem of Saxton and Thomason \cite{saxton2015hypergraph}.

An attractive corollary of Theorem \ref{thm:main} is the Szemer\'edi's theorem in random subsets of integers.

\begin{theorem} \label{thm:szemeredi_random}
    For every $\alpha > 0$ and integer $k \ge 3$, there exists $C > 0$ such that if $p \ge Cn^{-1/(k-1)}$ then a set $S \subseteq [n]$ formed by taking each number from $[n]$ with probability $p$, independently, with high probability has the property that every $S' \subseteq S$ of size $|S'| \ge \alpha |S|$ contains a $k$-term AP.
\end{theorem}

We leave easy derivation of Theorem \ref{thm:szemeredi_random} from Theorem \ref{thm:main} as an exercise (alternatively, see \cite[Corollary 1.2]{balogh2015independent}). The case $k=3$ of Theorem \ref{thm:main} was first settled by Kohayakawa, \L uczak, and R\"odl \cite{kohayakawa96} (see \cite{samotij2015independent} for a short proof) and the full statement was proven in breakthroughs by Conlon and Gowers \cite{conlon2016combinatorial} and, independently, by Schacht \cite{schacht2016extremal}. 

The purpose of this paper is to give a short, inductive proof of Theorem \ref{thm:main}.

\section{Proof}

Given sets $S \subseteq D \subseteq [n]$ and integers $0 \le k' \le k$, we denote with $\AP_{k',k}(S, D)$ the number of $k$-term APs in $D$ which contain at least $k'$ elements from $S$. Similarly, given $x \in [n]$, let $\AP_{k',k}(x, S, D)$ denote the number of $k$-term APs in $D$ which contain $x$ and at least $k'$ elements from $S \setminus \{x\}$.

We prove the following strengthening of Theorem \ref{thm:main}.

\begin{theorem} \label{thm:main_ind}
    For every $k, k' \in \mathbb{N}$, where $0 \le k' \le k$ and $k \ge 3$, and every $\beta, \alpha > 0$, there exist $\gamma, \xi, C > 0$ and $n_0$ such that the following holds for $n \ge n_0$ and $m \ge C n^{1 - 1/(k - 1)}$: For any $D \subseteq [n]$ of size $|D| \ge \alpha n$, all but at most $\beta^m \binom{n}{m}$ $m$-element subsets $S \subseteq D$ have the property that $\AP_{k',k}(S \setminus X, D) \ge \gamma n^2 (m/n)^{k'}$ for every for every $X \subseteq S$ of size $|X| \le \xi m$.
\end{theorem}

Theorem \ref{thm:main_ind} with $k = k'$ implies Theorem \ref{thm:main}. Both the statement of Theorem \ref{thm:main_ind} and the proof are largely based on the ideas of Schacht \cite[Theorem 3.3]{schacht2016extremal}.

The proof of Theorem \ref{thm:main_ind} contains two ingredients. One is the difficult Szemer\'edi's theorem \cite{szemeredi75}, and the other is the following lemma, obtained by a standard application of the R\"odl-Ruci\'nski deletion method \cite[Lemma 4]{rodl1995threshold}. For the sake of completeness, we give a proof in the appendix.

\begin{lemma} \label{lemma:rodl_rucinski_deletion}
Given integers $1 \le k' < k$ and $\beta, \xi > 0$, there exist $C, T > 0$ such that the following holds for $n$ sufficiently large and $m \ge Cn^{1 - 1/(k-1)}$: For all but at most $\beta^m \binom{n}{m}$ $m$-element sets $S \subseteq [n]$, there exists $X \subset S$ of size $|X| \le \xi m$ such that
$$
    \frac{1}{n} \sum_{x \in [n]} \AP_{k',k}(x, S \setminus X, [n])^2 \le T \mu_e^2,
$$
where $\mu_e = n (m/n)^{k'}$.
\end{lemma}

The need for removing a set $X$ in order to get a bound on the sum in the previous lemma hints why we need a stronger, `robust' version of Theorem \ref{thm:main}.

\begin{proof}[Proof of Theorem \ref{thm:main_ind}]
We prove the theorem by double induction on $k$ and $k'$. For $k' = 0$, we just need that $D$ contains $\gamma n^2$ $k$-term arithmetic progressions, which follows from Szemer\'edi's theorem and a simple averaging argument by Varnavides \cite{varnavides59}. Suppose that it holds for $k' - 1$ (as $k'$) for some $k' > 0$. We show that then it also holds for $k'$. Without loss of generality, we can assume $m < n/2$.

Let $\lambda = \beta^6 / (100e)^3$. We shall invoke the induction hypothesis with $k' - 1$ (as $k'$) and $\lambda$ (as both $\beta$ and $\alpha$). Let $\gamma', \xi', C$ be constants corresponding to these parameters. Furthermore, we invoke Lemma \ref{lemma:rodl_rucinski_deletion} with $\lambda$ (as $\beta$) and $\xi'/2$ (as $\xi$), and let $T > 0$ be the corresponding constant. Set
$$
    z = \lceil \gamma' / (4T) \rceil, \quad \gamma = \gamma' k / (2 \cdot (6z)^{k'}), \quad \xi = \xi'/(12z).
$$
We prove the theorem for $m$ divisible by $2z$, which is easily seen to imply the statement for every other $m$. The following values will be of interest:
$$
    \mu = n^2 (m/n)^{k'}, \quad m' = m/(2z), \quad \mu_e = k n (m'/3n)^{k'-1}.
$$

Given a (small) subset $S \subseteq D$, let us denote with $\cT(S)$ the set of `saturated' elements as follows:
$$
    \cT(S) = \left\{ x \in D \colon \AP_{k'-1,k}(x, S, D) \ge \gamma' \mu_e / 2 \right\}.
$$
The following definition encompasses the `advancing' argument used in the proof.

\textbf{Advancing property $\cA(\hat S)$.} Given $\hat S \subseteq D$, we say that $S \subseteq D$ is \emph{advancing} with respect to $\hat S$ if, for every $X \subseteq S$ of size $|X| \le \xi' m' / 6$, at least one of the following holds:
\begin{enumerate}[(A)]
    \item \label{prop:A} $\AP_{k',k}(\hat S \cup (S \setminus X), D) \ge \gamma \mu$,
    \item \label{prop:B} $|\cT(\hat S \cup (S \setminus X ))| > |\cT(\hat S)| + n/z$.
\end{enumerate}

\begin{claim} \label{claim:main_counting_claim}
Given $\hat S \subseteq D$, all but at most $(\beta/4)^{2m'} \binom{n}{m'}$ $m'$-element subsets $S \subset D$ satisfy $\cA(\hat S)$.
\end{claim}
We postpone the proof of the claim until the end.

The following concept is the heart of our proof: Given a subset $Z \subseteq \{1, \ldots, 2z\}$, a sequence of $m'$-element subsets $\mathbf{S} = (S_1, \ldots, S_{2z})$ and a sequence of subsets $\mathbf{X} = (X_i)_{i \in Z}$, where for each $i \in Z$ we have $X_i \subseteq S_i \subseteq D$ and $|X_i| \le \xi' m' / 6$, we say that $\mathbf{S}$ is \emph{$(Z, \mathbf{X})$-bad} if, for each $i \in [2z] \setminus Z$, $S_i \notin \cA(\hat S_{i-1})$ where
\begin{equation} \label{eq:hat_S}
    \hat S_{i} = \bigcup_{\substack{j \in Z\\j \le i}} S_j \setminus X_j.
\end{equation}
We say that $\mathbf{S}$ is \emph{bad} if it is $(Z, \mathbf{X})$-bad for some choice of $Z$ and $\mathbf{X}$ with $|Z| \le z$.

Next, we show that if a set $S \subseteq [n]$ of size $|S| = m = 2zm'$ can be partitioned into $\mathbf{S} = (S_1, \ldots, S_{2z})$ which is not bad, then $S$ satisfies the property of Theorem \ref{thm:main_ind}. To this end, suppose that $\mathbf{S}$ is such a partition. Consider some $X \subseteq S$ of size $|X| \le \xi m = \xi' m' / 6$. Set $Z = \emptyset$, and for each $i \in [2z]$, sequentially, do the following: if $S_i \in \cA(\hat S_{i-1})$ then add $i$ to $Z$ and set $X_i = X \cap S_i$, where $\hat S_{i-1}$ is as defined in \eqref{eq:hat_S} with respect to the current set $Z$. As $\mathbf{S}$ is not bad, we have $|Z| > z$. Without loss of generality, suppose $\{1, \ldots, z+1\} \subseteq Z$. If $\AP_{k'}(\hat S_{z+1}) < \gamma \mu$, then by the property \ref{prop:B} we have
$$
    |\cT(\hat S_{i+1})| > |\cT(\hat S_i)| + n/z
$$
for every $i \in \{1, \ldots, z\}$. However, then $|\cT(\hat S_{z+1})| > n \ge |D|$, which cannot be.

Note that each subset $S$ of size $|S| = 2zm'$, which does not satisfy the property of the theorem, gives exactly $m! / (m'!)^{2z}$ unique bad sequences. Our aim is to show that there are at most $(\beta/2)^m \binom{n}{m'}^{2z}$ bad sequences $\mathbf{S}$, which implies there are at most
$$
    (\beta/2)^m \binom{n}{m'}^{2z} \cdot \frac{(m'!)^{2z}}{m!} < (\beta/2)^m \frac{n^{m}}{m!} \stackrel{(m<n/2)}{<} \beta^m \binom{n}{m}
$$
`bad' subsets $S$. The number of bad sequences can easily be counted using Claim \ref{claim:main_counting_claim}. First, we can choose the witnessing $Z$ of size $|Z| = z' \le z$ in at most $2^{2z}$ ways. For a fixed choice of $Z$, we choose sets $S_i$ and a witnessing $X_i$, for $i \in Z$, in at most $\binom{n}{m'}^{z'} 2^{m' z'}$ ways. For each other $i \in [2z] \setminus Z$, there are at most $(\beta/4)^{2m'} \binom{n}{m'}$ ways to choose $S_i \notin \cA(\hat S_{i-1})$. As $|[2z] \setminus Z| \ge z$ we have $m' (2z - z') \ge m/2$, thus putting all the factors together gives at most $(\beta/2)^m \binom{n}{m'}^{2z}$ bad sequence, as required.
\end{proof}

It remains to prove Claim \ref{claim:main_counting_claim}.

\begin{proof}[Proof of Claim \ref{claim:main_counting_claim}]
For any $m'$-element set $S \subseteq D$ which contains at least $2m'/3$ elements from $\cT(\hat S)$, we have
\begin{equation} \label{eq:prop_a_count}
    \AP_{k',k}(\hat S \cup (S \setminus X), D) \ge (m'/3) \cdot \gamma' \mu_e / 2 \ge \gamma \mu
\end{equation}
for any $X \subset S$ of size $|X| \le m'/3$. If $|D \setminus \cT(\hat S)| < \lambda n$, then the number of $m'$-element subsets $S \subseteq D$ which do not have this property, that is which contain at least $t \ge m'/3$ elements from $D \setminus \cT(\hat S)$, is at most
$$
    \sum_{t \ge m'/3} \binom{\lambda n}{t} \binom{n}{m' - t} < m \binom{\lambda n}{m'/3}\binom{n}{2m'/3} < (3e \lambda^{1/3})^{m'} \binom{n}{m'} < (\beta/4)^{2m'} \binom{n}{m'},
$$
and the part \ref{prop:A} holds for all other sets.

Suppose now that $|D \setminus \cT(\hat S)| \ge \lambda n$. By the observation leading to \eqref{eq:prop_a_count}, it suffices to consider those $S$ which contain at least $m'/3$ elements from $D' = D \setminus \cT(\hat S)$. By the induction hypothesis applied with $k' - 1$ (as $k'$), $\lambda$ (as both $\alpha$ and $\beta$) and $D'$ (as $D$) --- giving $\gamma', \xi'$ --- and by Lemma \ref{lemma:rodl_rucinski_deletion} applied with $\xi'/2$ (as $\xi$) and $\lambda$ (as $\beta$), for all but at most $2 \lambda^{m'/3} \binom{n}{m'/3}$ sets $S' \subseteq D'$ of size $m'/3$ the following holds: There exists a set $X' \subseteq S'$ of size $|X'| \le \xi' m' / 6$ such that, for any $X \subseteq S'$ of size $|X| \le \xi' m' / 6$, we have
\begin{equation} \label{eq:first_moment_sum}
    \AP_{k'-1,k}(S'_X, D') \ge \frac{\gamma'}{k} n \mu_e
\end{equation}
and
\begin{equation} \label{eq:second_moment}
    \frac{1}{n} \sum_{x \in [n]} \AP_{k'-1,k}(x, S'_X, D')^2 \le T \mu_e^2,
\end{equation}
where $S'_X = S' \setminus (X \cup X')$. The constants were chosen such that $|X \cup X'| \le \xi' |S'|$. Therefore, at most
$$
    2 \lambda^{m'/3} \binom{n}{m'/3} \binom{n}{2m'/3} < (\beta/4)^{2m'} \binom{n}{m'}
$$
$m'$-element subsets $S \subseteq D$ with $|S \cap D'| \ge m'/3$ do not contain such a subset $S'$. We show that all other sets $S$ satisfy the part \ref{prop:B}.

From \eqref{eq:first_moment_sum} we get
\begin{equation} \label{eq:first_moment}
    \frac{1}{n} \sum_{x \in [n]} \AP_{k'-1}(x, S'_X, D') = \AP_{k'-1}(S'_X, D') k / n \ge \gamma' \mu_e.
\end{equation}
Let $Y = \AP_{k',k}(x, S_x', D')$ for $x \in [n]$ chosen uniformly at random. From \eqref{eq:first_moment} and \eqref{eq:second_moment} we conclude
$$
\mathbb{E}[Y] \ge  \gamma' \mu_e \quad \text{and} \quad \mathbb{E}[Y^2] \le T \mu_e^2,
$$
thus by the Paley-Zygmund inequality, we have 
$$
    \Pr\left(Y \ge \frac{1}{2}\mathbb{E}[Y] \right) \ge \frac{\mathbb{E}[Y]^2}{4 \mathbb{E}[Y^2]} \ge \gamma' / (4T) > 1 / z.
$$
In particular, there are at least $n/z$ elements $x \in D'$ such that $\AP_{k'-1,k}(x, S'_X, D') \ge \gamma' \mu_e / 2$. As $D' \cap \cT(\hat S) = \emptyset$, each such element belongs to $\cT(\hat S \cup (S \setminus X))$ and, by the definition of $D'$, not to $\cT(\hat S)$, thus \ref{prop:B} holds.
\end{proof}

{\small \bibliographystyle{abbrv} \bibliography{main}}

\appendix
\section{Proof of Lemma \ref{lemma:rodl_rucinski_deletion}}

The following result is an analogue of \cite[Lemma 4]{rodl1995threshold}.

\begin{lemma} \label{lemma:deletion_general}
    Let $Y$ be a set with $N$ elements, and $\cS$ be a family of $s$-element subsets of $Y$, for some integers $s$ and $N$. Let $k < N/4$ be an integer and $T > 1$. Then, for all but at most $$
        T^{-q/s} \binom{N}{m}
    $$
    subsets $Y_m \subseteq Y$ of size $m$, there exists a subset $X \subseteq Y_m$ of size $|X| \le q$ such that $Y_m \setminus X$ contains at most 
    $$
        2^s T |\cS|(m/N)^{s}
    $$
    sets from $\cS$.
\end{lemma}

\begin{proof}[Proof of Lemma \ref{lemma:rodl_rucinski_deletion}]
Let $\cB$ be the family of $k'$-element subsets $B \subseteq [n]$ which are contained in a $k$-term arithmetic progression. Note that for each $B$ there are at most $K$ such APs, for some constant $K$ depending only on $k$. Furthermore, we write $B \sim B'$ for $B, B' \in \cB$ if there exist $k$-APs $B \subseteq A$ and $B' \subseteq A'$ such that $A \cap A'$ contains an element which does not belong to $B \cup B'$. With this notation at hand, for a given set $S \subseteq [n]$ we have
\begin{equation} \label{eq:sum_AP}
    \sum_{x \in [n]} \AP_{k',k}(x, S, [n])^2 \le K^2 \sum_{\substack{B \sim B'\\B, B' \in \cB}} \mathbf{1}(B \subseteq S) \cdot \mathbf{1}(B' \subseteq S) = K^2 \sum_{\substack{B \sim B'\\B, B' \in \cB}} \mathbf{1}(B \cup B' \subseteq S),
\end{equation}
where $\mathbf{1}(\cdot)$ is an indicator function. Let
$$
    \cB' = \{ B \cup B' \colon B, B' \in \cB \text{ and } B \sim B'\}, 
$$
and note that we can further upper bound \eqref{eq:sum_AP} as
$$
   2 \cdot 3^{2k'} \cdot K^2 \sum_{\hat B \in \cB'} \mathbf{1}(\hat B \subseteq S).
$$
The first factor of $2$ takes into account that we count $B \sim B'$ and $B' \sim B$ separately, and $3^{2k'}$ is an upper bound on the number of ways we can represent $\hat B$ as $B \cup B'$. Let $\cB'_s \subseteq \cB'$ denote sets which are of size exactly $s$, for $s \in \{k', \ldots, 2k'\}$. By Lemma \ref{lemma:deletion_general}, for all but at most
\begin{equation} \label{eq:beta_final}
    \sum_{s=k'}^{2k'} T^{-q/s} \binom{n}{m} < 2k' T^{-q/(2k')} \binom{n}{m}
\end{equation}
$m$-element subsets $S$, for each $s$ there exists $X_s \subseteq S$ of size $|X_s| < q = \xi' m$ such that $S \setminus X_s$ contains at most
\begin{equation} \label{eq:TBs}
    2^s T |\cB_s'| (m/n)^s 
\end{equation}
sets from $\cB_s'$. 

If $s < 2k'$, then by retracing the definition of $B \sim B'$ we see that $B \cup B'$ belongs to an arithmetic progression, thus $|\cB'_s| = O(n^2)$. As $n^2 (m/n)^{k'} = o(n^3 (m/n)^{2k'})$, the contribution coming from $\cB_{s}'$ for $s < 2k'$ is negligible. For $s = 2k'$ we have $|\cB_{2k'}'| = O(n^3)$, thus \eqref{eq:TBs} is of order $O(n^3 (m/n)^{2k'})$, which is exactly the bound we are aiming for. By choosing $\xi'$ to be sufficiently small for a given $\xi$, and correspondingly increasing $T$, the lemma follows.
\end{proof}

\end{document}